\documentclass[12pt]{amsart}

\newcommand{\T}{\mathbb{T}} 
\newcommand{\C}{\mathbb{C}} 
\newcommand{\D}{\mathbb{D}} 
\newcommand{\ints}{\mathbb{Z}} 
\newcommand{\cpd}{\overline{\D^d}} 
\newcommand{\al}{\alpha}
\newcommand{\ip}[2]{\langle #1, #2 \rangle}
\newcommand{\mcp}{\mathcal{P}}
\newcommand{\Sz}{\mathcal{S}}

\numberwithin{equation}{section}

\newtheorem{theorem}{Theorem}[section]
\newtheorem{prop}[theorem]{Proposition}

\newtheorem{lemma}[theorem]{Lemma}
\newtheorem{definition}[theorem]{Definition}
\newtheorem{notation}[theorem]{Notation}
\newtheorem{question}[theorem]{Question}

\theoremstyle{remark}

\title[Kernel decompositions for Schur functions]{Kernel decompositions for Schur functions on the polydisk}

\author{Greg Knese}

\address{University of California, Irvine, Irvine, CA 92697-3875}
\date{\today} 
\email{gknese@uci.edu} 


\usepackage{amsmath, amsthm, amssymb}
\usepackage{ifthen}
\begin{document}
\bibliographystyle{plain}

\maketitle

\begin{abstract}
A certain kernel (sometimes called the Pick kernel) associated to
Schur functions on the disk is always positive semi-definite.  A
generalization of this fact is well-known for Schur functions on the
polydisk.  In this article, we show that the ``Pick kernel'' on the
polydisk has a great deal of structure beyond being positive
semi-definite.  It can always be split into two kernels possessing
certain shift invariance properties.
\end{abstract}

\section{Introduction}
Let $\D^d$ be the unit polydisk in $\C^d$.  A \emph{Schur function} is
simply a holomorphic function $f: \D^d \to \C$ bounded by one in
modulus.  One of the most fundamental facts about Schur functions in
one variable is that the following kernel is positive semi-definite:
\begin{equation} \label{schur-1var}
\frac{1-f(z) \overline{f(\zeta)}}{1-z\bar{\zeta}} \geq 0.
\end{equation}
(We say a function $K:\D^d \times \D^d \to \C$ is a positive
semi-definite kernel and write $K \geq 0$ if for every finite subset
$F \subset \D^d$, the matrix
\[
(K(z,\zeta))_{z,\zeta \in F}
\]
is positive semi-definite---to actually form a matrix we would need an
ordering of $F$, but this is unimportant).

The positive semi-definiteness of \eqref{schur-1var} is significant
because (1) it relates function theory to operator theory and (2) it
turns out to have a very strong converse: if $f$ is a function on a
finite subset of $\D$ such that \eqref{schur-1var} is positive
semi-definite on that finite set, then $f$ is the restriction of a
Schur function.  This is the content of the Pick interpolation
theorem. 

It is not clear what the ``best'' generalization of \eqref{schur-1var}
is to several variables.  For a Schur function in $d$ variables, it is
a fact that
\begin{equation} \label{schur-dvar}
\frac{1-f(z)\overline{f(\zeta)}}{\prod_{j=1}^d (1-z_j\bar{\zeta_j})} 
\end{equation}
is positive semi-definite, however this does not seem to be extremely
useful.  Here $z = (z_1,\dots, z_d), \zeta = (\zeta_1,\dots, \zeta_d)
\in \C^d$.

It was not until circa 1988 that a more useful result was given in two
variables by J.Agler \cite{jA88}: for any Schur function $f$ on $\D^2$
there exist positive semi-definite kernels $\Gamma_1, \Gamma_2: \D^2
\times \D^2 \to \C$ such that
\begin{equation} \label{agler-decomp}
1-f(z) \overline{f(\zeta)} = (1-z_1 \bar{\zeta_1})\Gamma_1(z,\zeta) +
(1-z_2 \bar{\zeta_2})\Gamma_2(z,\zeta).
\end{equation}
This formula, called an \emph{Agler decomposition}, does not
generalize to more variables in the way that its form suggests.  Schur
functions which satisfy
\begin{equation} \label{agler-decomp2}
1-f(z) \overline{f(\zeta)} = \sum_{j=1}^d (1-z_j\bar{\zeta_j})
\Gamma_j(z,\zeta)
\end{equation}
for some positive semi-definite kernels $\Gamma_1,\dots, \Gamma_d$,
form a proper subclass of the set of Schur functions called the
Schur-Agler class.  

Very recently, A. Grinshpan, D. Kaliuzhnyi-Verbovetskyi, V. Vinnikov,
and H. Woerdeman \cite{GKVW09} proved a decomposition that \emph{does}
hold in general and which is still analogous to
\eqref{agler-decomp}.  We state it in the scalar valued case but it
holds in the operator valued case as well.

\begin{theorem}[GKVW 2009 \cite{GKVW09}] \label{GKVWthm}
 Let $f: \D^d \to \D$ be holomorphic.  Then, for each $j \ne k \in
 \{1,\dots, d\}$ there exist positive semi-definite kernels $K$ and
 $K'$ such that
\[
1-f(z)\overline{f(\zeta)}= \prod_{r \ne j} (1-z_r\bar{\zeta_r}) K(z,z) + \prod_{r \ne k}
(1-z_r \bar{\zeta_r}) K'(z,z)
\]
\end{theorem}

It is our goal to strengthen this theorem and to alter the point of
view slightly.  Rather than looking for more decompositions analogous
to \eqref{agler-decomp}, we instead attempt to illuminate the
structure of the kernel in \eqref{schur-dvar}.

Before presenting our theorem we need the following
definition.

\begin{definition} 
If $K$ is a positive semi-definite kernel on $\D^d$, we shall say $K$
is $z_j$-contractive or just \emph{$j$-contractive} if
\[
(1-z_j\bar{\zeta_j})K(z,\zeta) \geq 0.
\]
If $S \subset \{1,\dots, d\}$, then we say a kernel $K$ is
$S$-contractive if it is $j$-contractive for all $j \in S$.
\end{definition}

\begin{theorem} \label{schurthm}
Let $d\geq 2$ and let $f: \D^d \to \D$ be holomorphic. Then, for each
nonempty $S \subsetneq \{1,2,\dots, d\}$, there exist positive
semi-definite $S$-contractive kernels $K_S, L_S$, such that if $S
\sqcup T = \{1,\dots, d\}$ is a nontrivial partition, then
\[
\frac{1-f(z)\overline{f(\zeta)}}{\prod_{j=1}^d (1-z_j\bar{\zeta_j})} =
K_S(z,\zeta) + L_T(z,\zeta),
\]
\[
K_T -L_T = K_S-L_S \geq 0,
\]
and if $S \subset S' \subset \{1,2,\dots, d\}$ then
\[
K_S \geq K_{S'}.
\]
\end{theorem}

Kernel inequalities like the last line should be interpreted as saying
$K_S - K_{S'}$ is positive semi-definite.

The proof of Theorem \ref{GKVWthm} in \cite{GKVW09} amounts to the
case where $S$ is a singleton, however many of the decompositions
provided by Theorem \ref{schurthm} can be used to reprove Theorem
\ref{GKVWthm}.

Indeed, let $S\sqcup T = \{1,\dots, d\}$ be any partition with $j \in
S$ and $k \in T$.  Theorem \ref{GKVWthm} follows from writing as in
Theorem \ref{schurthm}
\begin{multline*}
1-f(z) \overline{f(\zeta)} = \prod_{r \ne j}(1-z_r\bar{\zeta_r})
(\underset{K(z,\zeta)}{\underbrace{(1-z_j \bar{\zeta_j})K_S(z,\zeta))}} \\
+ \prod_{r \ne
  k}(1-z_r\bar{\zeta_r}) (\underset{K'(z,\zeta)}{\underbrace{(1-z_k\bar{\zeta_k})
    L_T(z,\zeta))}}
\end{multline*}
and $K$ and $K'$ are positive since $K_S$ is $j$-contractive and $L_T$
is $k$-contractive.

Our proof of Theorem \ref{schurthm} relies on proving the result first
for rational inner functions continuous on $\cpd$; these can be
characterized as follows.

 Let $p \in \C[z] = \C[z_1,\dots, z_d]$ have no zeros on the closed
 polydisk $\cpd$ and suppose $\deg{p} \leq n = (n_1,\dots,n_d)$.
 Define
\begin{equation} \label{def-ptilde}
\tilde{p}(z) := z^n \overline{p(1/\bar{z})} = z_1^{n_1}\cdots z_d^{n_d}
\overline{p(1/\bar{z_1},\dots, 1/\bar{z_d})}
\end{equation}
(and notice $|p| = |\tilde{p}|$ on the $d$-torus $\T^d$).

Every regular rational inner function can be represented as $f(z) =
\tilde{p}(z)/p(z)$ for some choice of $p$ and some choice of $n \geq
\deg(p)$ as above (see Rudin \cite{wR69} Theorem 5.2.5).  We state a
theorem below describing the structure of the following kernel
\[
\mcp (z,\zeta) := \frac{p(z)\overline{p(\zeta)} -
  \tilde{p}(z)\overline{\tilde{p}(\zeta)}}{ \prod_{j=1}^d
  (1-z_j\bar{\zeta_j})},
\]
a trivial modification of \eqref{schur-dvar} in the case of $f =
\tilde{p}/p$.

First, we need another definition.

\begin{definition} \label{def-pkernel}
Let us call $K(z,\zeta):\D^d\times \D^d \to \C$ a \emph{$\mcp$-kernel}
if
\begin{itemize}
\item $\mcp \geq K \geq 0$ in the sense of kernels and
\item whenever $\mcp (z,\zeta) \geq f(z) \overline{f(\zeta)}$ and
  $K(z,\zeta) \geq \epsilon f(z)\overline{f(\zeta)}$ for some
  $\epsilon>0$, then we necessarily have $K(z,\zeta) \geq f(z)
  \overline{f(\zeta)}$.
\end{itemize}
\end{definition}

See Lemma \ref{lem-pkernel} in the Appendix for a precise description
of what this means.  The (aesthetic) point here is that we have a
theorem which does not refer to our methods of proof.  The follow
theorem is similar to Theorem \ref{schurthm} but more precise.

\begin{theorem} \label{mainthm}
  Let $p \in \C[z]$ be as above.  For every nonempty $S \subsetneq
  \{1,2,\dots, d\}$, there exist $S$-contractive $\mcp$-kernels $K_S,
  L_S$, such that if $S \sqcup T = \{1,\dots, d\}$ is a nontrivial
  partition, then
\[
\mcp = K_S + L_T.
\]
Moreover, $K_S$ is maximal among all $S$-contractive kernels bounded
above by $\mcp$.
\end{theorem}

This last condition makes these decompositions unique.

\section{The kernel $\mcp$}
The theorems from the introduction are proved by analyzing
orthogonality relations for a ``Bernstein-Szeg\H{o} measure'':
\begin{equation} \label{def-mu}
d\mu = \frac{1}{|p (z) |^2} d\sigma(z)
\end{equation}
where $d\sigma$ is normalized Lebesgue measure on the $d$-torus $\T^d$
and $p \in \C [z]$ has no zeros on the closed polydisk $\cpd$.  We
also use $d\sigma$ to represent normalized Lebesgue measure on
different dimensional tori, and the dimension will be made apparent by
the variable; e.g. $d\sigma(z_1)$ corresponds to normalized Lebesgue
measure on $\T$ using the variable $z_1$.

Notice that the complex Hilbert space $L^2(\mu)$ is a renorming of
$L^2(\T^d)$ and therefore is topologically isomorphic.  The inner
product on $L^2(\mu)$ is denoted
\[
\ip{f}{g}_{\mu} = \int_{\T^d} f(z)\overline{g(z)} d\mu(z).
\]

For a subset $X$ of the lattice $\ints^d$ we define the closed
subspace
\[
L^2_\mu(X) := \{ f \in L^2(\mu): \hat{f}(\al) = 0 \text{ for } \al
\notin X\}
\]
where $\hat{f}(\al)$ denotes the $\al$-th Fourier coefficient of $f$
(and note we typically use $\al=(\al_1,\dots, \al_d)$ to denote a
$d$-tuple of integers).  We use the following non-traditional
notation.  If $Y \subset X \subset \ints^d$ then we write
\begin{equation} \label{note-comple}
L^2_{\mu}(X\ominus Y) := L^2_{\mu}(X) \ominus L^2_{\mu}(Y).
\end{equation}

We use the following partial order on $d$-tuples of integers $\al =
(\al_1,\dots,\al_d), \beta=(\beta_1,\dots, \beta_d)$:
\[
\al \leq \beta \text{ if and only if } \al_j \leq \beta_j \text{ for
  all } j =1,\dots, d;
\]
$n = (n_1,\dots, n_d)$ is a fixed $d$-tuple which bounds the
multi-degree of $p$ (i.e. the degree of $p$ with respect to $z_j$ is
at most $n_j$); writing $\al < \beta$ means $\al \leq \beta$ and $\al
\ne \beta$.

We typically write elements of $\C^d$ with $z = (z_1,\dots,
z_d)$.  We use multi-index notation:
\[
z^\al := z_1^{\al_1} \cdots z_d^{\al_d}
\]
for $\al \in \ints^d$ and $z \in \C^d$.

We need to define various subsets of $\ints^d$:
\begin{align}
\ints_{+}^d &:= \{\al \in \ints^d: \al\geq 0\} \nonumber \\
\ints_{n+}^d &:= \{\al \in \ints^d: \al \geq n\} \nonumber \\
B &:= \ints_{+}^d \setminus \ints_{n+}^d = \{ \al \in \ints_{+}^d:
\exists j: \al_j < n_j \} = \{\al \in \ints_{+}^d: \al \ngeq n\} \label{def-B}
\end{align}

Then, for example $L^2_\mu(\ints_+^d)$ denotes the closure of the
polynomials with respect to $L^2(\mu)$, a space equal to the Hardy
space $H^2(\T^d)$ although it has a different inner product.

The first thing we prove provides the connection to the kernel $\mcp$.
See \cite{AM02} for background on reproducing kernel Hilbert spaces.
The \emph{Szeg\H{o} kernel} will be denoted:
\[
\Sz_d(z,\zeta) = \prod_{j=1}^{d} \frac{1}{(1-z_j\bar{\zeta_j})}.
\]
As $H^2(\T^d)$ is a reproducing kernel Hilbert space kernel $\Sz_d$
and since $L^2_{\mu}(\ints_{+}^d)$ is a renorming of $H^2(\T^d)$,
$L^2_{\mu}(\ints_{+}^d)$ and all of its closed subspaces are also
reproducing kernel Hilbert spaces.

\begin{prop} \label{prop-kernel} 
Let $p \in \C[z]$ have degree at most $n$, let $\tilde{p}(z) =
z^n\overline{p(1/\bar{z})}$, and let
\[
d\mu = \frac{1}{|p (z) |^2} d\sigma(z).
\]
Then, with $B$ as in \eqref{def-B} the reproducing kernel for
$L^2_\mu(B)$ is
\[
\mcp(z,\zeta) = (p(z)\overline{p(\zeta)} -
\tilde{p}(z)\overline{\tilde{p}(\zeta)}) \Sz_d(z,\zeta).
\]
\end{prop}

\begin{proof} The kernel for $L^2_\mu(\ints_{+}^{d})$
 is $p(z)\overline{p(\zeta)} \Sz_d(z,\zeta)$.  This is a simple
 computation; if $f \in H^2(\T^d)$ and $\zeta \in \D^d$ then
\[
\begin{aligned}
\int_{\T^d} f(z) \overline{p(z) \overline{p(\zeta)}
  \Sz_d(z,\zeta)} d\mu(z) &= \int_{\T^d} f(z)
\overline{p(z)}p(\zeta)\overline{\Sz_d(z,\zeta)}
\frac{d\sigma(z)}{|p(z)|^2} \\
&= \int_{\T^d} \frac{f(z)}{p(z)} p(\zeta) \overline{\Sz_d(z,\zeta)} d\sigma(z)
\\
&= \frac{f(\zeta)}{p(\zeta)}p(\zeta) = f(\zeta)
\end{aligned}
\]
The third equality is the reproducing property of $\Sz_d$ (or just the
Cauchy integral formula).

 We prove in Lemma \ref{lem-tilde} below that $L^2_\mu(\ints^d_{+}
 \ominus B)=\tilde{p}L^2_{\mu}(\ints_{+}^d)$ and a computation similar
 to that above proves that the reproducing kernel of
 $\tilde{p}L^2_{\mu}(\ints_{+}^d)$ is $\tilde{p}(z)
 \overline{\tilde{p}(\zeta)} \Sz_d(z,\zeta)$.  The result then follows
 from the fact that:
\[
L^2_{\mu}(\ints_{+}^d) = L^2_{\mu}(B) \oplus L^2_\mu(\ints^d_{+}
 \ominus B)
\]
and that the reproducing kernel of a direct sum is the sum of the
reproducing kernels of each direct summand. Namely,
\[
\underset{\text{kernel for }
  L^2_{\mu}(\ints_{+}^d)}{\underbrace{p(z)\overline{p(\zeta)}
    \Sz_d(z,\zeta)}} - \underset{\text{kernel for }
  L^2_\mu(\ints^d_{+} \ominus B)}{\underbrace{\tilde{p}(z)
    \overline{\tilde{p}(\zeta)} \Sz_d(z,\zeta)}} = \text{kernel for }
L^2_{\mu}(B).
 \]

\end{proof}

The following lemma was used above.

\begin{lemma} \label{lem-tilde} 
\[
\tilde{p} L^2_\mu(\ints_{+}^d) = L^2_\mu(\ints^d_{+} \ominus B) =
L^2_\mu(\ints^d\ominus(\ints^d \setminus \ints^d_{n+}))
\]
\end{lemma}

\begin{proof}

Observe that $\tilde{p}(z) = z^n \overline{p(z)}$ on $\T^d$ and so
\[
\begin{aligned}
\ip{z^\al}{\tilde{p}}_\mu &= \int_{\T^d} z^\al \overline{\tilde{p}(z)}
\frac{1}{|p(z)|^2} d\sigma(z) \\ & = \int_{\T^d} z^{\al}
\frac{\bar{z}^{n}p(z)}{|p(z)|^2} d\sigma(z) \\ & = \int_{\T^d}
\frac{z^{\al-n}}{\overline{p(z)}} d\sigma(z).
\end{aligned}
\]
This equals zero if any component of $\al-n$ is negative (i.e. if $\al
\ngeq n$) since $1/\bar{p}$ is anti-analytic in $\D^d$.  In
particular, if $\al \ngeq n$, then for $\beta \geq 0$, $\al \ngeq n +
\beta$ and therefore
\[
\ip{z^\al}{z^\beta \tilde{p}}_{\mu} = 0.
\]
This shows
\[
\tilde{p} L^2_{\mu}(\ints_{+}^{d}) \perp
L^2_\mu(\ints^d\setminus \ints^d_{n+})
\]
which means
\[
\tilde{p} L^2_{\mu}(\ints_{+}^{d}) \subset
L^2_\mu(\ints^d\ominus(\ints^d\setminus \ints^d_{n+})) \cap
L_\mu^2(\ints^d_+ \ominus B).
\]
  Conversely, if $f \in L_\mu^2(\ints^d_+ \ominus B)$ and $f \perp
  \tilde{p} L^2_{\mu}(\ints_{+}^{d})$, then we can show $f \perp
  L^2_\mu(\ints_{+}^d)$ as follows.

Since $p(0) \ne 0$, $\tilde{p}(z) = a z^n + q(z)$ with $a =
\overline{p(0)} \ne 0$ and $q$ of degree at most $n$ with no $z^n$
term.  By assumption on $f$, $f \perp \tilde{p}$ and $f \perp q$ ($q
\in L^2_{\mu}(B)$).  Therefore, $f \perp z^n$.  From here we can give
an inductive proof on the lattice $\ints^d_{+}$.  If $f$ is orthogonal
to all non-negative frequencies less than some $\al \geq n$, then $f$
is orthogonal to
\[
z^{\al-n}\tilde{p}(z) = a z^\al + z^{\al-n} q(z) \qquad \text{ and  } \qquad
z^{\al-n}q(z)
\]
as the latter contains only frequencies less than $\al$.  This implies
$f \perp z^\al$, and by induction $f \perp L^2_\mu(\ints_+^{d})$. (As
this is a non-traditional way of doing induction we should explain
using the contrapositive: if $f$ is \emph{not} perpendicular to some
$z^\al$, then $f$ must also \emph{not} be perpendicular to some
$z^\beta$ with $\beta < \al$.  This can be continued until $f$ is
\emph{not} perpendicular to a monomial supported in $B$---a
contradiction.) This forces $f \equiv 0$.

Hence, $L^2_\mu (\ints^d_{+} \ominus B) = \tilde{p}
L^2_\mu(\ints_{+}^d) \subset L^2_{\mu}(\ints^d\ominus
(\ints^d\setminus \ints_{n+}^d))$.  By Lemma \ref{lem-comple} given
below, we automatically have
\[
L^2_\mu (\ints^d_{+} \ominus B) = \tilde{p} L^2_\mu(\ints_{+}^{d})
= L^2_{\mu}(\ints^d\ominus \ints_{n+}^d).
\]  
\end{proof}

\begin{lemma}\label{lem-comple}
 Suppose $W,Y \subset \ints^d$ and set $X = W \cup Y$.  Then,
\begin{equation} \label{cond1}
L^2_{\mu} (X) \ominus L^2_{\mu}(Y) \subset L^2_{\mu}(W)
\end{equation}
if and only if
\begin{equation} \label{cond2}
L^2_{\mu} (W) \ominus L^2_{\mu}(Y\cap W) \subset
(L^2_{\mu}(Y))^{\perp}
\end{equation}
and in either case 
\[
L^2_{\mu}(X) \ominus L^2_{\mu}(Y) = L^2_{\mu}(W)\ominus
L^2_{\mu}(Y\cap W).
\]
\end{lemma}

\begin{proof}
This is essentially a result of the decomposition
\begin{align}
L^2_{\mu}(X\ominus (Y\cap W)) &= L^2_{\mu}(X \ominus Y) \oplus
L^2_{\mu}(Y\ominus (Y\cap W)) \label{decomp1}\\
&= L^2_{\mu}(X \ominus W) \oplus
L^2_{\mu}(W\ominus (Y\cap W)). \label{decomp2}
\end{align}

Suppose $L^2_{\mu}(X\ominus Y) \subset L^2_{\mu}(W)$ which necessarily
means $L^2_{\mu}(X \ominus Y) \subset L^2_{\mu}(W \ominus (Y\cap W))$.
If $f \in L^2_{\mu}(W \ominus (Y\cap W)) \ominus L^2_{\mu}(X \ominus
Y)$, then $f \in L^2_{\mu}(Y \ominus (Y\cap W))$ by \eqref{decomp1}.
Hence, $f \in L^2_{\mu}(Y\cap W \ominus Y\cap W) = \{0\}$ showing that
$L^2_{\mu}(X\ominus Y)$ fills out all of $L^2_{\mu}(W \ominus (Y\cap
W))$.  

Suppose $L^2_{\mu}(W \ominus (Y\cap W)) \subset L^2_{\mu} (Y)^{\perp}$
which necessarily means $L^2_{\mu}(W \ominus (Y\cap W)) \subset
L^2_{\mu}(X \ominus Y)$.  If $f \in L^2_{\mu}(X \ominus Y) \ominus
L^2_{\mu}(W \ominus (Y\cap W))$, then $f \in L^2_{\mu}(X \ominus W)$
by \eqref{decomp2}.  Hence, $f \perp L^2_{\mu}(Y) + L^2_{\mu}(W) =
L^2_{\mu}(X)$, forcing $f \equiv 0$. This shows that $L^2_{\mu}(W
\ominus (Y\cap W))$ fills out all of $L^2_{\mu}(X\ominus Y)$.
\end{proof}

So, we have shown that $\mcp$ represents the reproducing kernel of
$L^2_{\mu} (B)$.  Any orthogonal decomposition of $L^2_{\mu}(B)$ then
gives a decomposition of $\mcp$.  Our goal is to prove that
$L^2_{\mu}(B)$ has a decomposition with very special properties.

\section{Orthogonal decompositions of $L^2_{\mu} (B)$}
We recall the definition of $B$ and define several subsets of $B$ below:
\begin{notation} \label{Bsets}
\[
\begin{aligned}
X_j &:= \{\al \in \ints^{d}_{+}: \al_j < n_j\} \\ X_S &:=
\bigcup_{j\in S} X_j = \{\al \in \ints^{d}_{+}: \exists j \in S: \al_j
< n_j\} \\ B & = \bigcup_{j=1}^{d} X_j = \{ \al \in \ints_{+}^d:
\exists j: \al_j < n_j\} \\
\end{aligned}
\]
where $S \subset \{1,2,\dots,d\}$.
\end{notation}

\begin{prop} \label{prop-orth}
With the same setup as Proposition \ref{prop-kernel} let $S\sqcup T
=\{1,2,\dots, d\}$ be a partition. Then,

\[
L^2_{\mu}(B) = L^2_{\mu}(X_S) \oplus L^2_{\mu}(X_T \ominus (X_T\cap
X_S)).
\] 
\end{prop}

The content of the above proposition is that the subspaces listed in
the orthogonal decomposition are \emph{actually orthogonal}, something
which would not hold for a general finite measure on $\T^d$. This
proposition is still valid if $S$ or $T$ are empty if we interpret
$X_{\varnothing} = \{0\}$.  This makes the proposition sensible
(although trivial) in the case $d=1$ (something useful later).

We need the following notation for use in dividing up all of
structures according to the partition $S\sqcup T = \{1,\dots,
d\}$. There is no harm in assuming $S = \{1,\dots s\}$, $T =
\{s+1,\dots, d\}$, and $t:= d-s$.
\[
\begin{aligned}
z_S &= (z_1,\dots, z_s) \in \C^{s}, & z_T &= (z_{s+1}, \dots, z_d) \in
\C^t, & z &= (z_S, z_T)\\
n_S &= (n_1,\dots, n_s) \in \ints^{s} , & n_T &= (n_{s+1} ,\dots, n_d) \in \ints^{t}, &
n &= (n_S,n_T) \\
\al_S &= (\al_1,\dots, \al_s) \in \ints^{s} , & \al_T &= (\al_{s+1},
\dots, \al_d) \in \ints^{t}, & \al &= (\al_S, \al_T) \\
B_S &= \{ \al_S \in \ints_{+}^{s}: \al_S \ngeq n_S\}, & B_T &= \{\al_T
\in \ints_{+}^{t}: \al_T \ngeq n_T\} &
\end{aligned}
\]

\begin{proof}[Proof of Proposition \ref{prop-orth}]
The proposition is really a type of inclusion-exclusion principle as
it can be rewritten as saying
\[
L^2_{\mu}((X_S\cup X_T) \ominus X_S) = L^2_{\mu}(X_T\ominus (X_S\cap
X_T))
\]
since $B = X_S\cup X_T$.

To prove it, consider following the measures $\mu_{z_S}$ on $\T^t$
which are indexed by $z_S \in \T^s$:
\[
d\mu_{z_S}(z_T) = \frac{1}{|p(z_S,z_T)|^2} d\sigma(z_T)
\]
i.e. for each $z_S \in \T^s$ we get a measure on $\T^t$, and points in
$\T^t$ are denoted by $z_T$.  

By Proposition \ref{prop-kernel}, the reproducing kernel for
$L^2_{\mu_{z_S}}(B_T)$ is
\[
\mcp^T_{z_S}(z_T,\zeta_T) := (p(z_S,z_T)\overline{p(z_S,\zeta_S)}
- \tilde{p}(z_S, z_T)
\overline{\tilde{p}(z_S,\zeta_T)})\Sz_t(z_T,\zeta_T)
\]
where again $\Sz_t$ is the $t$-dimensional Szeg\H{o} kernel.  Notice
that $\mcp^T_{z_S}(z_T,\zeta_T)$ is a trigonometric polynomial of
degree at most $n_S$ as a function of $z_S$, while as a function of
$z_T$ this function only has Fourier coefficients corresponding to
points of $B_T$.  For these reasons, the function of $z =(z_S,z_T) \in
\T^d$ defined for each fixed $\zeta \in \D^d$ by
\[
L_{\zeta}(z) = L(z,\zeta) = z_S^{n_S} \bar{\zeta_S}^{n_S}
 \Sz_s(z_S,\zeta_S) \mcp^T_{z_S}(z_T,\zeta_T)
\]
is in $L^2_{\mu}(\ints_{+}^s \times B_T) =
L^2_{\mu}(X_T)$. (Specifically, as a function of $(z_S,z_T)$
\[
\Sz_s(z_S,\zeta_S) \in L^2_{\mu}(\ints_{+}^s \times
  \{0_T\})
\]
\[
\mcp^T_{z_S}(z_T,\zeta_T) \in L^2_{\mu}([-n_S, n_S] \times B_T)
\]
Here $0_T$ is the zero $t$-tuple in $\ints^t$ and $[-n_S,n_S] = \{\al_S
\in \ints^{s}: -n_S\leq \al_S \leq n_S\}$.)
So, if $f \perp L^2_{\mu}(X_T)$, then
\begin{equation} \label{obs1}
\ip{f}{L_\zeta}_{\mu} = 0 \text{ for all } \zeta \in \D^d.
\end{equation}

On the other hand, $L$ can be thought of as a difference
of two terms:
\begin{multline*}
L_\zeta(z)=\underset{A_{\zeta}}{\underbrace{p(z_S,z_T)\overline{p(z_S,\zeta_T)} (z_S^{n_S}\bar{\zeta_S}^{n_S})
\Sz_d(z,\zeta)}} \\ - \underset{B_\zeta}{\underbrace{\tilde{p}(z_S,z_T)\overline{\tilde{p}(z_S,\zeta_T)}
 (z_S^{n_S}\bar{\zeta_S}^{n_S}) \Sz_{d}(z,\zeta)}}.
\end{multline*}
(We used $\Sz_d(z,\zeta) = \Sz_s(z_S,\zeta_S) \Sz_t(z_T,\zeta_T)$
above.)

Since $z_S^{n_S} \overline{p(z_S,\zeta_T)}$ has only non-negative
Fourier coefficients in $z_S$, the second term $B_\zeta$ is an element
of $\tilde{p} H^2(\T^d) = L^2_{\mu}(\ints_{+}^{d} \ominus B)$. So, if
$f \in L^2_{\mu}(B)$, then $B_\zeta \perp f$ and we have
\begin{equation} \label{obs2}
\ip{f}{L_\zeta}_{\mu} = \ip{f}{A_\zeta}_{\mu}.
\end{equation}

Finally, if $f \in L^2_{\mu}(\ints_{+}^d)$ then
\begin{align}
\ip{f}{A_\zeta}_{\mu} &= \int_{\T^s} \int_{\T^{t}} f(z)
\overline{p(z)} p(z_S,\zeta_T) \overline{\Sz_{t}(z_T,\zeta_T)}
\frac{d\sigma(z_T)}{|p(z)|^2}
(\bar{z_S}^{n_S}\zeta_S^{n_S})\overline{\Sz_s(z_S,\zeta_S)}
d\sigma(z_S) \nonumber \\ &= \int_{\T^s} \int_{\T^{t}}
\frac{f(z)}{p(z)} p(z_S,\zeta_T) \overline{\Sz_{t}(z_T,\zeta_T)}
d\sigma(z_T)
(\bar{z_S}^{n_S}\zeta_S^{n_S})\overline{\Sz_s(z_S,\zeta_S)}
d\sigma(z_S) \nonumber \\ 
&=\int_{\T^s}
\frac{f(z_S,\zeta_T)}{p(z_S,\zeta_T)} p(z_S,\zeta_T)
(\bar{z_S}^{n_S}\zeta_S^{n_S})\overline{\Sz_s(z_S,\zeta_S)}
d\sigma(z_S) \nonumber \\
&= \int_{\T^s} f(z_S,\zeta_T)
(\bar{z_S}^{n_S}\zeta_S^{n_S})\overline{\Sz_s(z_S,\zeta_S)} d\sigma(z_S)
\nonumber \\ &
= \sum_{\al_S \geq n_S} \sum_{\al_T \geq 0} \hat{f}(\al_S,\al_T)
\zeta^\al \label{obs3}
\end{align}
which is the $L^2(\T^d)$ projection of $f$ to $z_S^{n_S}H^2(\T^d)$.
(The second and fourth equalities are algebra, the third is the
reproducing property of $\Sz_t$, and the fifth is a Fourier series
computation.) 

If we combine the observations \eqref{obs1}, \eqref{obs2},
\eqref{obs3} above we see that if
\[
f \perp L^2_{\mu}(X_T) \text{ and } f \in L^2_{\mu}(B)
\]
then 
\[
\hat{f}(\al) = 0
\]
for $\al_S\geq n_S$, $\al_T \geq 0$ and therefore $f \in
L^2_{\mu}(X_S)$. So, $L^2_{\mu}(B\ominus X_T) \subset L^2_{\mu}(X_S)$.

By Lemma \ref{lem-comple}, this proves
\[
L^2_{\mu}(B\ominus X_T) = L^2_{\mu}(X_S\ominus (X_S\cap X_T))
\]
since $B = X_S \cup X_T$.
\end{proof}

\section{Closed under shifts}
The goal of this section is to prove two facts.

\begin{prop} \label{prop-inv0} 
With the setup of Proposition \ref{prop-orth}, $L^2_{\mu}(X_S)$ is
closed under multiplication by $z_j$ for all $j \notin S$, and contains
all subspaces of $L^2_{\mu}(B)$ with this property.
\end{prop}

\begin{prop} \label{prop-inv} With the setup of Proposition
  \ref{prop-orth}, $L^2_{\mu} (X_S\ominus
  (X_S\cap X_T))$ is closed under multiplication by $z_j$ for all $j
  \in T$.  
\end{prop}

The first fact is not difficult.
\begin{proof}[Proof of Proposition \ref{prop-inv0}]
An element $f \in L^2_{\mu}(B)$ is in $L^2_{\mu}(X_S)$ if and only if
$\hat{f}(\al) = 0$ whenever $\al_k \geq n_k$ for all $k \in S$.  This
property is obviously unaffected by multiplying $f$ by variables $z_j$
for $j \notin S$.  

On the other hand, if $f \in L^2_{\mu}(B)$, has the property that 
\[
z^\al f \in L^2_{\mu}(B)
\]
for all $\al \geq 0$ satisfying $\al_j = 0$ for $j\in S$, then $f$
must be an element of $L^2_{\mu}(X_S)$.  Otherwise, $\hat{f}(\al) \ne
0$ for some $\al \geq 0$, with $\al_k \geq n_k$ for all $k\in S$.
But then if we set $m = (m_1,\dots, m_d)$ where
\[
m_j = \begin{cases} 0 & \text{for } j \in S \\ 
n_j & \text{for } j \notin S \end{cases}
\]
then $z^m f \notin L^2_{\mu}(B)$---a contradiction.  This proves that
$L^2_{\mu}(X_S)$ contains all subspaces closed under multiplication by
all $z_j$ for $j \notin S$.
\end{proof}

As for Proposition \ref{prop-inv}, it is convenient to prove the
proposition by adjoining a variable and using results in $d$ variables
that have already been proven.  Elements of $\C^{d+1}$ will be written
as $(z_0,z)$. So, now $p \in \C[z_0,z]$ is a polynomial of $d+1$
variables of degree at most $(n_0,n)$ with no zeros in
$\overline{\D^{d+1}}$.  The measure $\mu$ corresponds to
$|p(z_0,z)|^{-2} d\sigma(z_0,z)$.

 Notation already defined for $d$ variables will retain its meaning,
 while we will use the following notation for certain $d+1$-variable
 objects:
\[
\begin{aligned}
Y_j &= \{(\al_0,\al) \in \ints_{+}^{d+1}: \al_j < n_j\} \\
Y_S &= \bigcup_{j \in S} Y_j \text{ for } S \subset \{0,1,\dots,d\}
\end{aligned}
\]

We also find it convenient to use interval notation for subsets of
integers (as opposed to real numbers):
\[
\begin{aligned}
(a,b) &= \{k\in \ints: a < k < b\} \\
[a,b) &= \{ k \in \ints: a\leq k < b\}, \text{ etc..} \\
\end{aligned}
\]
We never make use of intervals of real numbers, so there should be no
confusion. 

Now, let $S\sqcup T$ be a partition of $\{1,\dots, d\}$, and let $T_0
= T\cup\{0\}$.  We will prove that
\[
L^2_{\mu}(Y_{S} \ominus (Y_{T_0} \cap Y_S))
\]
is closed under multiplication by $z_0$.  This is enough to prove the
proposition.

\begin{proof}[Proof of Proposition \ref{prop-inv}]
For each $z_0 \in \T$, let
$d\mu_{z_0}(z)$ be the measure on $\T^d$
\[
d\mu_{z_0}(z) = \frac{1}{|p(z_0,z)|^2} d\sigma(z).
\]

Let 
\[
\Gamma_{z_0} (z,\zeta)
\]
denote the reproducing kernel for $L^2_{\mu_{z_0}}(X_T \ominus
(X_T\cap X_S))$, 
and let
\[
\Delta_{z_0}(z,\zeta)
\]
denote the reproducing kernel for $L^2_{\mu_{z_0}}(X_S)$.

By Proposition \ref{prop-orth},
\begin{multline*}
(p(z_0,z)\overline{p(z_0,\zeta)} - \tilde{p}(z_0,z)\overline{
  \tilde{p}(z_0,\zeta)}) \Sz_{d}(z,\zeta)\\
= \Gamma_{z_0}(z,\zeta) +
\Delta_{z_0}(z,\zeta).
\end{multline*}

The left hand side is a trigonometric polynomial in $z_0$ of degree at
most $n_0$, while $\Delta_{z_0}(z,\zeta)$ as a function of $z$ is the
only function on the right hand side with any Fourier support in $X_S
\setminus X_T$. This means the coefficients of $z^\al$ in
$\Delta_{z_0}$ for $\al \in X_S\setminus X_T$ are trig polynomials
with respect to $z_0$; i.e.
\begin{align} \label{deltasupport}
\Delta_{z_0}(z,\zeta)  \in & L^2_{\mu}(\ints\times (X_S\cap X_T) \cup
      [-n_0,n_0]\times (X_S\setminus X_T)) \\
&= L^2_{\mu}(\ints\times (X_S\cap X_T) \cup
      [-n_0,n_0]\times X_S) \nonumber
\end{align}

(Perhaps it needs to be explicitly stated that $\Delta_{z_0}(z,\zeta)$
is actually in $L^2(\T^{d+1})$ as a function of $z=(z_0,z) \in
\T^{d+1}$.  See Lemma \ref{regularity} below.)

Define for each $Z=(\zeta_0,\zeta) \in \D^{d+1}$
\begin{equation} \label{def-L}
L_Z (z_0, z) = L((z_0,z),Z) = \frac{\bar{z_0} \zeta_0}{1-\bar{z_0}\zeta_0}
\Delta_{z_0} (z,\zeta).
\end{equation}
By \eqref{deltasupport} and \eqref{def-L}, 
\[
L_Z \in L^2_{\mu}(\ints \times (X_S\cap
X_T) \cup (-\infty,n_0) \times X_S). 
\]

Now, let $f \in L^2_{\mu}(\ints\times X_S)$, then for each $Z =
(\zeta_0, \zeta) \in \D^{d+1}$
\begin{align}
\ip{f}{L_Z}_{\mu} &= \int_{\T} \int_{\T^d} f(z_0,z) \overline{\Delta_{z_0}
  (z,\zeta)} d\mu_{z_0}(z) \frac{z_0\bar{\zeta_0}}{1-z_0\bar{\zeta_0}}
d\sigma(z_0) \label{Lcomp1}\\ &= \int_{\T} f(z_0,\zeta)
\frac{z_0\bar{\zeta_0}}{1-z_0\bar{\zeta_0}}
d\sigma(z_0) \label{Lcomp2}\\ &= \sum_{\al_0=-\infty}^{-1} \sum_{\al
  \geq 0} \hat{f}(\al_0,\al) \bar{\zeta_0}^{-\al_0}
\zeta^\al; \label{Lcomp3}
\end{align}
the equality \eqref{Lcomp1} is by definition, \eqref{Lcomp2} is
because $\Delta_{z_0}$ is a reproducing kernel for $X_S$ with respect
to $\mu_{z_0}$, and \eqref{Lcomp3} is a Fourier series computation.
If
\[
\begin{aligned}
f &\in L^2_{\mu}(\ints \times X_S)  \text{ and }\\
f & \perp L^2_{\mu}(\ints \times (X_S\cap X_T) \cup (-\infty,n_0) \times
X_S)
\end{aligned}
\]
then $f \perp L_Z$ and therefore the expression in \eqref{Lcomp3} is
zero which implies $f \in L^2_{\mu}(\ints_{+} \times X_S) =
L^2_{\mu}(Y_S)$.  Hence, by Lemma \ref{lem-comple}
\begin{equation} \label{eq768}
L^2_{\mu}(\ints \times X_S) \ominus L^2_{\mu}(\ints \times (X_S\cap
X_T) \cup (-\infty,n_0) \times X_S)
\end{equation}
is unchanged if we intersect all sets with $Y_S$. This proves
\eqref{eq768} equals
\begin{equation} \label{eq774}
 L^2_{\mu}(Y_S)
\ominus L^2_{\mu}(Y_S \cap Y_{T_0})
\end{equation}
where we are using the facts that
\[
(\ints \times X_S) \cap Y_S = Y_S
\]
and
\[
\begin{aligned}
&(\ints \times (X_S\cap X_T) \cup (-\infty,n_0) \times X_S) \cap Y_S 
\\ &= (Y_S\cap Y_T) \cup (Y_S \cap Y_{\{0\}})
\\ &= Y_S\cap Y_{T_0}.
\end{aligned}
\]
This proves
\[
L^2_{\mu}(Y_S \ominus (Y_S\cap Y_{T_0})) \perp L^2_{\mu}((-\infty,0)
\times X_S)
\]
since \eqref{eq768} $=$ \eqref{eq774} and since
\[
(-\infty, 0) \times X_S \subset \ints \times (X_S\cap X_T) \cup
(-\infty,n_0) \times (X_S\setminus X_T).
\]
This is enough to show $L^2_{\mu}(Y_S\ominus (Y_S\cap Y_{T_0}))$
is closed under multiplication by $z_0$, as follows.

Let $f \in L^2_{\mu}(Y_S\ominus (Y_S\cap Y_{T_0}))$.  By Proposition
\ref{prop-inv0}, it is clear that $z_0f \in L^2_{\mu}(Y_S)$.  To show
$z_0f \perp L^2_{\mu}(Y_S \cap Y_{T_0})$, let $(\al_0,\al) \in Y_S\cap
Y_{T_0}$. If $\al_0 >0$ then $(\al_0-1,\al) \in Y_S\cap Y_{T_0}$ in
which case
\begin{equation} \label{whichcase}
\ip{z_0f}{z_0^{\al_0} z^{\al}}_{\mu} = \ip{f}{z_0^{\al_0-1}
  z^{\al}}_{\mu} = 0.
\end{equation}
If $\al_0 = 0$, then $(\al_0-1,\al) \in (-\infty,0)\times X_S$ in
which case we again have \eqref{whichcase} because $f \perp
L^2_{\mu}((-\infty,0) \times X_S)$. Hence, $z_0f \in
L^2_{\mu}(Y_S\ominus (Y_S\cap Y_{T_0}))$, proving that this subspace
is closed under multiplication by $z_0$.
\end{proof}

We used the following lemma in the above proof.

\begin{lemma} \label{regularity}
Let $X \subset \ints^{d}_{+}$, $\zeta \in \D^{d}$.  The reproducing
kernel of $L^2_{\mu_{z_0}} (X)$, written $K_{z_0}(X)(z,\zeta)$ is in
$L^2(\T^{d+1})$ as a function of $(z_0,z)$.
\end{lemma}

\begin{proof} 
For each $\al \in \ints^{d}$, let
\[
C_{\al}(z_0) = \int_{\T^{d}} \frac{z^\al}{|p(z_0,z)|^2}
d\sigma(z)
\]
and define the following (generally infinite) self-adjoint matrix
indexed by $X$
\[
C_X(z_0) = (C_{\al-\beta} (z_0))_{\al, \beta \in X}.
\]
The entries of $C_X(z_0)$ are clearly continuous on $\T$.
Since $|p|$ is bounded above and below on the circle, it turns out
$C_X(z_0)$ is bounded above and below as an operator on $\ell^2(X)$.
Indeed, for $(v_\al) \in \ell^2(X)$
\[
\sum_{\al, \beta \in X}
C_{\al-\beta}(z_0)v_\al \bar{v_\beta}  = \int_{\T^{d}} \frac{|\sum_{\al \in X}
  v_\al z^\al|^2}{|p(z_0,z)|^2} d\sigma(z)
\]
is bounded above and below by
\[
\int_{\T^{d}} |\sum_{\al \in X}
  v_\al z^\al|^2 d\sigma(z) = \sum_{\al \in X} |v_\al|^2
\]
with constants $c_1 = (\inf_{\T^{d+1}} |p|)^{-2}$ and $c_2= (\sup_{\T^{d+1}}
|p|)^{-2}$ respectively.

Let 
\[
B_{\al,\beta}(z_0) = (C_X(z_0))^{-1}_{\al,\beta}
\]
be the $(\al, \beta)$ entry of the inverse of $C_X(z_0)$.  The
reproducing kernel $K_{z_0}(X)(z,\zeta)$ can be given explicitly as
\[
K_{z_0}(X)(z,\zeta) = \sum_{\al, \beta \in X} B_{\beta,\al}(z_0)
z^\al (\bar{\zeta})^{\beta}.
\]
The proof of this fact is a direct computation; if $\gamma \in X$,
then
\[
\ip{z^\gamma}{\sum_{\al, \beta \in X} B_{\beta,\al}(z_0) z^\al
  (\bar{\zeta})^{\beta}}_{\mu_{z_0}} = \sum_{\al ,\beta \in X}
C_{\gamma - \al}(z_0) B_{\al,\beta}(z_0) \zeta^{\beta} =
\zeta^{\gamma}.
\]

Since $C_X(z_0)$ is bounded above and below, 
\[
\sum_{\al \in X} (\sum_{\beta \in X} B_{\al,\beta}(z_0)
(\bar{\zeta})^{\beta}) z^{\al}
\]
is in $L^2(\T^{d+1})$ as a function of $(z_0,z)$ for each $\zeta \in
\D^{d}$.
\end{proof}

\section{Proof of Theorem \ref{mainthm}}

So far we have shown (in Prop. \ref{prop-orth})
\[
L^2_{\mu}(B) = L^2_{\mu}(X_T) \oplus L^2_{\mu}(X_S \ominus (X_S\cap X_T))
\]
for each partition $S\sqcup T = \{1,\dots, d\}$.  In addition,
$L^2_{\mu}(X_S)$ and $L^2_{\mu}(X_S \ominus (X_S\cap X_T))$ are closed
under multiplication by all variables $z_j$ for $j \in T$ and
$L^2_{\mu}(X_S)$ is maximal among subspaces with this property
(Propositions \ref{prop-inv0} and \ref{prop-inv}). 

Theorem \ref{mainthm} now reduces to bookkeeping and facts about
reproducing kernels.  Namely, a kernel is a $\mcp$-kernel if it is the
reproducing kernel for a closed subspace of $L^2_{\mu}(B)$ (Lemma
\ref{lem-pkernel}).  For a nonempty $S \subset \{1,\dots, d\}$, set $T
= \{1,\dots, d\} \setminus S$ and let
\begin{itemize}
\item $K_S$ be the reproducing kernel
for $L^2_{\mu}(X_T)$ and 
\item $L_S$ be the reproducing kernel for
$L^2_{\mu}(X_T\ominus (X_S\cap X_T))$
\end{itemize}
(these definitions look like $S$ and $T$ have been mistakenly switched
but they have not).  Both $K_S$ and $L_S$ are $S$-contractive
$\mcp$-kernels by Lemma \ref{lem-jcont} and Propositions
\ref{prop-inv0} and \ref{prop-inv}.

By Proposition \ref{prop-orth} we have
\[
\mcp = K_S + L_T.
\]  

To prove the maximality property of $K_S$, suppose $\mcp \geq K \geq
0$ for some $S$-contractive kernel $K$.  By Lemmas
\ref{lem-sum} and \ref{lem-shift} below, $z^\al K_\zeta \in
L^2_{\mu}(B)$ for all $\zeta \in \D^d$ and all $\al \geq 0$ satisfying
$\al_j = 0$ for $j \notin S$.  By Proposition \ref{prop-inv0},
$K_\zeta \in L^2_{\mu}(X_T)$ and therefore by Lemma \ref{lem-dom},
$K_S$ must dominate $K$:
\[
K_S \geq K.
\]

This completes the proof of Theorem \ref{mainthm}.

\section{Proof of Theorem \ref{schurthm}}
We have already proven the theorem for rational inner functions which
are regular on $\cpd$, since such functions can always be represented
by $f = \tilde{p}/p$ where $p \in \C[z]$ with no zeros on $\cpd$.
Namely, we have by Theorem \ref{mainthm}
\[
\frac{1-f(z)\overline{f(\zeta)}}{\prod_{j=1}^{d} (1-z_j\bar{\zeta_j})}
= \frac{K_S(z,\zeta)}{p(z)\overline{p(\zeta)}} +
\frac{L_T(z,z)}{p(z)\overline{p(\zeta)}}.
\]
Let us agree to absorb the denominators into the definitions of $K_S$
and $L_T$ so that we really have the formula
\[
\frac{1-f(z)\overline{f(\zeta)}}{\prod_{j=1}^{d} (1-z_j\bar{\zeta_j})} =
K_S(z,\zeta) + L_T(z,\zeta).
\]
By Theorem \ref{mainthm}, $K_S + L_T = K_T + L_S$ and by maximality of
$K_S, K_T$ among $S$ and $T$-contractive $\mcp$-kernels, respectively,
we have
\[
K_S- L_S = K_T -L_T \geq 0
\]
and
\[
K_S \geq K_{S'} \text{ for } S \subset S'.
\]

To prove the theorem for a general holomorphic function $f: \D^d \to
\D$, we use a theorem of Rudin (\cite{wR69} Theorem 5.5.1) which says
that such $f$ can be approximated uniformly on compact subsets of
$\D^d$ by rational inner functions, regular on $\cpd$.  So, say $f_k
\to f$ uniformly on compacta, with each $f_k$ rational, inner, and
continuous up to $\cpd$. We have corresponding decompositions:
\[
\frac{1-f_k(z)\overline{f_k(\zeta)}}{\prod_{j=1}^{d}(1-z_j
  \bar{\zeta_j})} = K^{(k)}_S(z,\zeta) + L^{(k)}_T(z,\zeta).
\]
Since
\[
|K^{(k)}_S(z,\zeta)|^2 \leq K^{(k)}_{S}(z,z) K^{(k)}_{S}(\zeta,\zeta) \leq
\frac{1}{\prod_{j= 1}^{d} (1-|z_j|^2)(1-|\zeta_j|^2)}
\]
(with $L^{(k)}_T$ satisfying a similar estimate), we see that the
$K^{(k)}_S$'s and $L^{(k)}_T$'s are holomorphic on $\D^d\times \D^d$
and locally uniformly bounded and hence they are in a normal
family. Taking subsequences, we may assume $K^{(k)}_S$ converges to
some $K_S$ and $L^{(k)}_{T}$ converges to some $L_T$ locally
uniformly.  Positive semi-definiteness, $S$ and $T$ contractivity, and
the identities/inequalities
\[
K_S- L_S = K_T -L_T \geq 0
\]
\[
K_S \geq K_S' \text{ for } S \subset S'
\]
are all preserved under such limits.

Therefore we conclude that
\[
\frac{1-f(z)\overline{f(\zeta)}}{\prod_{j=1}^{d} (1-z_j\bar{\zeta_j})} =
K_S(z,\zeta) + L_T(z,\zeta)
\]
is a valid decomposition.
 
\section{Reproducing Kernel Appendix}
We record a number of facts about reproducing kernels which we used
above.  We are sketchy since much of this is well-known.  For general
references see \cite{AM02} and \cite{BB84}.  As before, $\mcp$ is the
reproducing kernel for $L^2_{\mu} (B)$, where $d\mu = |p|^{-2}d\sigma$
and $B = \{\al \in \ints_{+}^d: \al \ngeq n\}$. (These details are by
no means essential for what follows.)  

\begin{lemma} \label{lem-basic}
 A function $f:\D^d \to \C$ is in a reproducing kernel Hilbert
 function space $\mathcal{H}$ on $\D^d$ with kernel $K$ if and only if
\[
K(z,\zeta) \geq \epsilon f(z)\overline{f(\zeta)}
\]
for some $\epsilon >0$.  The largest possible $\epsilon$ is equal to
$||f||^{-2}$. 
\end{lemma}

See Theorem 2.2 in \cite{BB84}.

\begin{lemma} \label{lem-sum} 
Let $K$ be a positive semi-definite kernel on $\D^d$, and let $f$ be a
finite linear combination of functions of the form $K_{\eta}(z) :=
K(z,\eta)$.  Then, there is an $\epsilon > 0$ such that
\[
K(z,\zeta) \geq \epsilon f(z)\overline{f(\zeta)}.
\]
In the case of a single kernel function we can say 
\[
K(z,\zeta) \geq \epsilon K_\eta(z) \overline{K_{\eta} (\zeta)}
\]
if and only if $1 \geq \epsilon K(\eta,\eta)$.
\end{lemma}

\begin{proof} Omitted. \end{proof}

\begin{lemma} \label{lem-pchar} 
A positive semi-definite kernel $K$ satisfying $\mcp \geq K$ is a
$\mcp$-kernel (as in Definition \ref{def-pkernel}) if and only if for
every function $f:\D^d \to \C$
\[
K(z,\zeta) \geq \epsilon f(z)\overline{f(\zeta)}
\]
implies 
\[
K(z,\zeta) \geq \frac{f(z)\overline{f(\zeta)}}{||f||^2_{\mu}}
\]
in which case we necessarily have $||f||^{-2}_{\mu} \geq \epsilon$.
In particular, $K(\zeta,\zeta) = ||K_\zeta||_\mu^2$ holds for all
$\zeta \in \D^d$ whenever $K$ is a $\mcp$-kernel.  (Here $K_\zeta(z) =
K(z,\zeta)$.)
\end{lemma}  

\begin{proof} Follows from the definition of a $\mcp$-kernel and
  Lemma \ref{lem-basic}.
 \end{proof}

\begin{lemma} \label{lem-dom} Suppose $\mcp \geq K \geq 0$.  Let
  $\mathcal{H} = \vee \{K_{\zeta}: \zeta \in \D^d\}$ be the closed
  span in $L^2_{\mu}(B)$ of the functions $K_{\zeta}(z) = K(z,\zeta)$,
  and let $L$ be the reproducing kernel for $\mathcal{H}$.  Then,
  $L\geq K$.
\end{lemma}

\begin{proof} 
This essentially follows from Corollary 2.6 of \cite{BB84}.
\end{proof}

\begin{lemma} \label{lem-pkernel} $K$ is a $\mcp$-kernel if and only
  if $K$ is a reproducing kernel for a closed subspace of
  $L^2_{\mu}(B)$.
\end{lemma}

\begin{proof}  

The forward direction is not used in this article, so we
omit the proof.  The converse statement is not difficult; it follows
from Lemmas \ref{lem-basic} and \ref{lem-pchar} and the fact that the
norm on a subspace is the same as the norm in the original space.
\end{proof}

\begin{lemma} \label{lem-shift} 
If a kernel $K$ with $\mcp \geq K$ is $j$-contractive, then 
\[
K(z,\zeta) \geq \epsilon f(z)\overline{f(\zeta)}
\]
implies $f, z_jf \in L^2_{\mu}(B)$.
\end{lemma}

\begin{proof} By assumption, $(1-z_j\bar{\zeta_j})K(z,\zeta) \geq 0$
 and therefore
\[
\mcp(z,\zeta) \geq K(z,\zeta) \geq z_j \bar{\zeta_j} K(z,\zeta) \geq
\epsilon z_j \bar{\zeta_j} f(z)\overline{f(\zeta)}
\]
which shows $z_j f \in L^2_{\mu}(B)$.
\end{proof}

\begin{lemma} \label{lem-jcont}
 If $\mathcal{H}$ is a closed subspace of $L^2_{\mu}(B)$ and
 $\mathcal{H}$ is closed under multiplication by $z_j$, then the
 reproducing kernel for $\mathcal{H}$ is $j$-contractive.
\end{lemma}

See for example Corollary 2.37 of \cite{AM02}.

\bibliography{schurkernel}

\end{document}